\documentclass{amsart}
 \usepackage{ulem}
\usepackage{tabularx}
\usepackage{enumitem,amsmath,amsfonts,amssymb,xcolor,latexsym,euscript,gensymb}
\usepackage[numbers]{natbib}

\newtheorem{lemma}{\bf Lemma}[section]

\newtheorem{prop}[lemma]{\bf Proposition}
\newtheorem{teor}[lemma]{\bf Theorem}

\newtheorem{cor}[lemma]{\bf Corollary}

\newtheorem{problem}[lemma]{\bf Problem}

\newcommand{\GL}{{\operatorname{GL}}}
\newcommand{\AGL}{{\operatorname{AGL}}}

\DeclareMathOperator{\Aut}{Aut}

\DeclareMathOperator{\soc}{soc}

\DeclareMathOperator{\GammaL}{\mathrm{\Gamma L}}
\DeclareMathOperator{\AGammaL}{\mathrm{A\Gamma L}}

\input xy
\xyoption{all}

\title[]{On finite groups with the Magnus Property}

\author{Martino Garonzi}
\address{Martino Garonzi. Departamento de Matem\'atica, Universidade de Bras\'ilia, Campus 
Universit\'ario Darcy Ribeiro, Bras\'ilia-DF, 70910-900, Brazil. \newline
ORCID: https://orcid.org/0000-0003-0041-3131}
\email{martino@mat.unb.br, mgaronzi@gmail.com}

\author{Claude Marion} 
\address{Claude Marion. Centro de Matem\'atica, Faculdade de Ci\^{e}ncias, Universidade do Porto, Rua do Campo Alegre 687, 4169-007 Porto, Portugal
\newline
ORCID: https://orcid.org/0000-0002-3802-1750}
\email{claude.marion@fc.up.pt}

\thanks{The first author acknowledges the support of Conselho Nacional 
de Desenvolvimento Cient\'ifico e Tecnol\'ogico (CNPq), Universal
- Grant number 402934/2021-0. 
The second author acknowledges the support from the Centre of Mathematics of the University of
Porto which is financed by national funds through the Funda\c{c}\~{a}o para a Ci\^{e}ncia e a Tecnologia,
I.P., under the project with references UIDB/00144/2020 and UIDP/00144/2020, as well as the Department of Mathematics of the University of Bras\'{i}lia for an invited research visit and its hospitality.}

\date{}


\begin{document}

\setlength{\parskip}{2mm}

\begin{abstract}
We investigate finite groups with the Magnus Property, where a group is said to have the Magnus Property (MP) if whenever two elements have the same normal closure then they are conjugate or inverse conjugate. In particular we observe that a finite MP group is solvable, determine the finite primitive MP groups and determine all the possible orders of the chief factors of a finite MP group. We also determine the MP finite direct products of finite primitive groups, as well as the MP crown-based powers of a finite monolithic primitive group. 
\end{abstract}

\maketitle

\section{Introduction}
Given a group $G$,  elements $x,y\in G$ and  a subset $S\subseteq G$, we let $x^y=y^{-1}xy$, $x^G$ denote the conjugacy class of $x$ in $G$, $C_G(x)$ be the centralizer of $x$ in $G$, and $\langle S\rangle$ and $\langle S\rangle^G$ denote respectively the subgroup of $G$ generated by $S$ and the normal closure of $S$ in $G$  which is the smallest normal subgroup of $G$ containing $S$. Note that $\langle x\rangle^G=\langle x^G\rangle$.

A group $G$ has the Magnus Property (MP) if whenever $x,y \in G$ are such that $\langle x^G \rangle = \langle y^G \rangle$, either $x$ is conjugate to $y$ or $x$ is conjugate to $y^{-1}$. We also say that a group $G$ has the Strong Magnus Property (SMP) if whenever $x,y \in G$ are such that $\langle x^G \rangle = \langle y^G \rangle$, $x$ is conjugate to $y$. In other words, an SMP group is an MP group in which every element is conjugate to its own inverse. Magnus' 
original motivation to study MP groups is, as he showed in \cite{Magnus30}, that free groups 
have this property. This can also be formulated by saying that if $G$ is any $1$-relator group realized as quotient of the free group $F$, say $G \cong F/\langle x \rangle^F$, then any other relator (on the same generators) realizing the same group must be either conjugate or inverse-conjugate to $x$.

We are interested in investigating finite groups which are MP.
As was noted in \cite{KMP}, a finite MP group is abelian if and only if all of its elements have order $1$, $2$, $3$, $4$ or $6$. It follows that every finite abelian MP group has the form $C_2^n \times C_3^m$ or $C_2^n \times C_4^m$, where $n,m$ are non-negative integers. 

Using the fact that a quotient of a finite MP group is again MP (see \cite[Corollary 2.5]{KMP}), our first observation is the following very strong structural property of finite MP groups.

\begin{prop} \label{solvability}
Every finite MP group is solvable.
\end{prop}

Denote by $\mathbb{F}_q$ the finite field of size $q$ and by $\mathbb{F}_q^n$ the $\mathbb{F}_q$-vector space of $n$-tuples of elements of $\mathbb{F}_q$. Recall that $M(9) = \mathbb{F}_3^2 \rtimes Q_8$, with the following identification of the quaternion group $Q_8= \langle i,j \rangle = \{\pm 1,\pm i,\pm j,\pm k\}$ as a subgroup of $\Aut(\mathbb{F}_3^2)=\GL(2,3)$ acting naturally on $\mathbb{F}_3^2$:
$$i \mapsto \left( \begin{array}{cc} 0 & 2 \\ 1 & 0 \end{array} \right), \hspace{1cm} 
j \mapsto \left( \begin{array}{cc} 1 & 1 \\ 1 & 2 \end{array} \right).$$
We also use the following standard notation: given some positive integers $m,n$ and a prime power $q$, $C_m$ denotes the cyclic group of order $m$, $S_m$, $A_m$ denote respectively the symmetric and alternating group of degree $m$, and $\AGL(n,q)$ and $\AGammaL(n,q)$ denote respectively the affine general and semilinear groups $\mathbb{F}_q^n \rtimes \GL(n,q)$ and $\mathbb{F}_q^n \rtimes \GammaL(n,q)$ where $\GammaL(n,q) \cong \GL(n,q) \rtimes \Aut(\mathbb{F}_q)$ is the group of invertible semilinear transformations of $\mathbb{F}_q^n$. When we write $C_7 \rtimes C_3$, we mean that the semidirect product is taken with respect to the nontrivial action.

Recall that the Frattini subgroup of a group $G$ is the intersection of all the maximal subgroups of $G$ if $G$ has maximal subgroups, else is defined to be $G$, and it is denoted by $\Phi(G)$. For us, a finite primitive group is a finite group $G$ admitting a maximal subgroup $M$ such that $\bigcap_{g \in G} g^{-1}Mg=\{1\}$. Our first main theorem is:

\begin{teor} \label{main1}
If $G$ is a nontrivial finite (S)MP group, then $G/\Phi(G)$ is isomorphic to a subdirect product of finite primitive (S)MP groups. The finite primitive MP groups are $C_2$, $C_3$, $S_3$, $A_4$, $\AGL(1,5)$, $C_7 \rtimes C_3$, $\AGL(1,7)$, $M(9)=C_3^2 \rtimes Q_8$. The finite primitive SMP groups are $C_2$, $S_3$ and $M(9)$.
\end{teor}

Actually, we prove the following stronger result.
\begin{teor} \label{main1bis}
Let $G$ be a finite primitive group of the form $V \rtimes G_0$ with $V$ elementary abelian and $G_0$ acting irreducibly on $V$. Assume $G_0$ is MP and any two elements of $V$ are either conjugate or inverse-conjugate in $G$. Then either $G$ is MP, in which case $G$ is one of the groups listed in Theorem \ref{main1}, or $G$ is isomorphic to one of the following groups: $\AGammaL(1,4) \cong S_4$, $\AGammaL(1,8)$, $\AGammaL(1,9)$, $\AGammaL(1,16)$. 
\end{teor}

Note that if $G$ if a finite primitive group satifying the hypotheses of Theorem \ref{main1bis} then $G_0\leqslant \GammaL(1,|V|)$. The two lists in Theorem \ref{main1} are the lists of all finite primitive MP and SMP groups, respectively. We point out that Theorem \ref{main1} has a couple of easy consequences: the first on the Fitting height of a finite MP group, the second on the possible prime divisors of the order of a finite MP group. 
The minimal length of a normal series of a finite solvable group $G$ with nilpotent factors is called the Fitting height of $G$ and it is denoted by $h(G)$.

\begin{cor} \label{Fitting}
If $G$ is a finite MP group, then $h(G) \leqslant 2$.
\end{cor}

\begin{cor} \label{consequences}
Let $G$ be a nontrivial finite MP group and let $\Pi$ be the set of prime divisors of $|G|$. The following facts are true.
\begin{enumerate}
\item $\Pi \subseteq \{2,3,5,7\}$.
\item If $G$ is nilpotent, then $\Pi \subseteq \{2,3\}$.
\item If $G$ is nilpotent and SMP, then $\Pi = \{2\}$, i.e. $G$ is a $2$-group.
\end{enumerate}
\end{cor}

A finite direct product of finite groups is SMP if and only if all the factors are (see Proposition \ref{SMPclass}). In the case of MP groups, the situation is more complicated. However we determine all the (S)MP finite direct products of finite primitive groups.
\begin{prop} \label{MPdir}
Let $n>1$ be an integer. For $1\leqslant i \leqslant n$ let $G_i$ be a finite primitive group and let $G:=G_1\times\ldots \times G_n$. The following assertions hold. 
\begin{enumerate}
\item If one of the $G_i$ is not MP then $G$ is not MP.
\item Suppose that $G_i$ is MP for each $i$. Then $G$ is MP if and only if, whenever $G_i\in\{C_7 \rtimes C_3, \mathrm{AGL}(1,5)\}$ for some $i \in \{1,\ldots,n\}$, then, for all $j \in \{1,\ldots,n\}$ with $j\neq i$, either $G_i=G_j=\mathrm{AGL}(1,5)$, or $G_j$ is SMP (i.e. $G_j\in\{C_2,S_3,M(9)\}$).
\end{enumerate} 
\end{prop}
Recall that a minimal normal subgroup of a group $G$ is a nontrivial normal subgroup of $G$ that does not properly contain any nontrivial normal subgroup of $G$, and the socle $\mathrm{soc}(G)$ of $G$ is the subgroup of $G$ generated by its minimal normal subgroups. 
A group is monolithic if it has a unique minimal normal subgroup.  Given a finite monolithic primitive group $L$ with socle $V$ and a positive integer $k$, the corresponding crown-based power of $L$ is
$$L_k=\{(l_1,\dots,l_k)\in L^k: l_1 \equiv \ldots \equiv l_k \mod V\}.$$
We fully characterize the (S)MP crowned-based powers of a finite monolithic primitive group:
\begin{prop}\label{MPcrownbased}
Let $L$ be a finite monolithic primitive group and let $k> 0$ be an integer. If $L \cong M(9)$, then $L_k$ is MP if and only if $k=1$, whereas if $L \not \cong M(9)$, then $L_k$ is (S)MP if and only if $L$ is (S)MP.
\end{prop}

Our remaining results are essentially on chief factors of a finite MP group. 
A chief factor of a group $G$ is a minimal normal subgroup of a quotient of $G$. 
Given a prime $p$, the $p$-rank of a finite group $G$ is the maximal dimension of a chief $p$-factor of $G$ and is denoted $r_p(G)$. If $G$ has no $p$-chief factor then $r_p(G)$ is set to be 0.
Given a finite solvable group $G$ and a prime $p$ dividing $|G|$, there is, by Hall's theorem, a largest positive integer $S_p(G)$ such that $G$ has a  maximal subgroup of index $p^{S_p(G)}$. If $p$ does not divide $|G|$ then $S_p(G)$ is set to be $0$.

Given a finite solvable group $G$ and a prime $p$ one has $S_p(G) = S_p(G/\Phi(G))$
and $S_p(G) \leqslant r_p(G)$ (see Lemma \ref{srineq}), 
however $r_p(G)$ may not equal $r_p(G/\Phi(G))$ and $S_p(G)$ may not equal $r_p(G)$. Let $j_p(G)$ be the largest integer $j$ such that $p^j$ appears as an index in a maximal chain of subgroups of $G$. Huppert \cite[Satz 1]{Huppert} proved that $r_p(G)=j_p(G)$. In case $G$ is a finite MP group, we  establish that  $r_p(G)=r_p(G/\Phi(G))=S_p(G)\leqslant 2$. 

\begin{teor} \label{thmprank}
Suppose that $G$ is a finite MP group. Let $p$ be a prime dividing $|G|$, so that $p \in \{2,3,5,7\}$. The following assertions hold. 
\begin{enumerate}
\item If $p\not \in \{2,3\}$ then $S_p(G)=r_p(G)=1$. Moreover, $G$ is supersolvable if and only if $G/\Phi(G)$ has no quotient isomorphic to $A_4$ nor $M(9)$.
\item If $p\in \{2,3\}$ then $S_p(G)\leqslant 2$, and $S_p(G)=1$ if and only if $G$ has no quotient isomorphic to $A_4$ if $p=2$, $M(9)$ if $p=3$.
\item If $p\in \{2,3\}$ then $r_p(G)\leqslant 2$.
\item We have $r_p(G)=r_p(G/\Phi(G))$.
\item We have $r_p(G)=S_p(G)$.
\end{enumerate}
\end{teor}

In the proof of Theorem \ref{thmprank} we use the concept of a pseudovariety of finite groups which is simply a class of finite groups closed under taking subgroups, quotients and finite direct products.

We finally obtain the following result characterizing the order of a chief factor of a finite (S)MP group:

\begin{cor}\label{c:orderchief}
If $G$ is an MP group then the order of any chief factor of $G$ is one of the following: $2$, $3$, $4$, $5$, $7$, $9$. Moreover, if $G$ is SMP then the order of any chief factor of $G$ is one of the following: $2$, $3$, $4$, $9$.
\end{cor}
Theorem \ref{main1} shows that, for each value $q$ in the lists appearing in Corollary \ref{c:orderchief}, there exists a finite (primitive) (S)MP group with a chief factor of order $q$.

We end this introduction with the following open problem:

\begin{problem}
Is it true that finite (S)MP groups have bounded derived length?
\end{problem}

The layout of the article is a follows. In Section \ref{generalproperties} we give some general properties of finite MP groups and show in particular that they are solvable, establishing Proposition \ref{solvability}. In Section \ref{primitiveMPgroups} we characterize the finite MP primitive groups and prove Theorem \ref{main1}. In Section \ref{dirprod} we characterize the MP finite direct products of finite primitive groups and the MP crown-based powers of a finite monolithic primitive group, proving Propositions \ref{MPdir} and \ref{MPcrownbased}.  Finally in Section \ref{sectionprank} we investigate some properties of the chief factors of a finite MP group and prove Theorem \ref{thmprank} and Corollary \ref{c:orderchief}.

\section{General properties of MP and SMP groups} \label{generalproperties}

A direct way to check whether a given finite group $G$ is MP is to run over all pairs $(x,y)$ of conjugacy class representatives of elements of $G$ and check whether it can happen that $\langle x^G \rangle = \langle y^G \rangle$ without being true that $x,y$ are either conjugate or inverse-conjugate. A much quicker way of doing this is the following. Set
$$\begin{array}{l}
A(G) := \{g^G \cup (g^{-1})^G\ :\ g \in G\}, \\
B(G) := \{\langle g^G \rangle\ :\ g \in G\}.
\end{array}$$
Then $G$ is MP if and only if $|A(G)|=|B(G)|$. This is algorithmically much quicker to implement since its complexity is roughly $k(G)$ (the number of conjugacy classes of $G$) while running over all pairs $(x,y)$ of conjugacy class representatives has complexity roughly $k(G)^2$.

The following result (and its proof) says that checking whether a given finite group is MP or SMP is easy to do using its character table.
\begin{prop} \label{chMP}
The Magnus and Strong Magnus Properties of a finite group $G$ can be detected from the character table of $G$.
\end{prop}
\begin{proof}
Let $\mathrm{Irr}(G)$ be the set of irreducible complex characters of $G$ and, for $x \in G$, let $N_x$ be the intersection of the normal subgroups $\ker(\chi)$, $\chi \in \mathrm{Irr}(G)$, containing $x$. Since every normal subgroup of $G$ is an intersection of kernels of irreducible complex characters of $G$, it follows that the normal closure $\langle x \rangle^G$ of $\langle x \rangle$ is precisely equal to $N_x$. Since two elements are conjugate if and only if their irreducible character values are equal and they are inverse-conjugate if and only if their irreducible character values are complex conjugate, the claim follows.
\end{proof}

As an application of Proposition \ref{chMP}, one can check that the semidirect product $C_7 \rtimes C_3$ (with nontrivial action) is MP (but not SMP). This also implies that the Magnus Property does not pass to subgroups, since $C_7$ is not MP. Similarly, the Strong Magnus Property is not closed under taking subgroups as $S_3$ is SMP but $C_3$ is not. However, in the case of finite groups, the Magnus and Strong Magnus Properties pass to quotients:

\begin{lemma} \label{quoMP}
If $G$ is a finite (S)MP group and $N \unlhd G$, then $G/N$ is (S)MP.
\end{lemma}
\begin{proof}
This follows from \cite[Corollary 2.5]{KMP} as well as \cite[Proposition 2.4]{KMP} and its proof when considering the SMP version of the lemma.
\end{proof}

We can now prove Proposition \ref{solvability}, i.e. the fact that every finite MP group is solvable. 

\begin{proof}[Proof of Proposition \ref{solvability}]
Assume $G$ is finite and MP. If $G$ is nonsolvable, then there is some $N \unlhd G$ such that $G/N$ has a minimal normal subgroup $L/N$ isomorphic to $S^m$ for some nonabelian simple group $S$. There exist two nontrivial elements in $L/N$ of distinct orders, and by minimality of $L/N$, their normal closure in $G/N$ is $L/N$, contradicting the fact that $G/N$ is MP (by Lemma \ref{quoMP}).
\end{proof}

Not every finite MP group is supersolvable, for example $M(9)$ is MP and not supersolvable. As we will see in Theorem \ref{thmprank}, $M(9)$ and $A_4$ are ``essentially'' the unique non-supersolvable finite MP groups, in the sense that any finite MP group that is not supersolvable must have $M(9)$ or $A_4$ as a quotient.

\section{Primitive MP groups} \label{primitiveMPgroups}

In order to better understand finite MP and SMP groups, we will use the notion of primitive group. Given a group $G$ and a subgroup $H$ of $G$, recall that the normal core of $H$ in $G$, denoted $H_G$, is the largest normal subgroup of $G$ contained in $H$, in particular $H_G = \bigcap_{g \in G} g^{-1}Hg$. Recall that a finite group $G$ is called primitive if it admits a maximal subgroup $M$ with trivial normal core $M_G = \{1\}$. If $G$ is any finite group and $M$ is any maximal subgroup of $G$, then the quotient $G/M_G$ is primitive, since $M/M_G$ is a maximal subgroup of $G/M_G$ with trivial normal core. If $M_1,\ldots,M_n$ are conjugacy class representatives of the maximal subgroups of a finite group $G$, then the kernel of the natural map $G \to \prod_{i=1}^n G/(M_i)_G$ equals the Frattini subgroup $\Phi(G)$, that is, the intersection of the maximal subgroups of $G$. This implies that $G/\Phi(G)$ is a subdirect product of primitive quotients of $G$. By Proposition \ref{quoMP}, we deduce that if $G$ is any finite MP group, then $G/\Phi(G)$ is a subdirect product of finite primitive MP groups. This means that it makes sense to study finite primitive MP groups.

Recall that the rank of a transitive permutation group $G$ acting on a set $\Omega$ is the number of orbits of its point stabilizer acting on $\Omega$. For example, a transitive permutation group is $2$-transitive if and only if it has rank $2$. Theorem \ref{main1} is a surprising result since it is a complete classification of all finite primitive MP (and SMP) groups. They are all $2$-transitive except $C_3$ and $C_7 \rtimes C_3$, which have rank $3$. It is worth noting the following immediate consequence of the theorem: except for $C_2$ and $C_3$, every finite primitive MP group is a Frobenius group, i.e. a semidirect product $N \rtimes H$ with the property that $C_N(h)=\{1\}$ for all $1 \neq h \in H$.

Recall that a finite collineation group of an affine line is a group of the form $V \rtimes G_0$ where $V$ is the additive group of a finite field of size $q$ and $G_0$ is a subgroup of $\GammaL(1,q) = \mathbb{F}_q^{\ast} \rtimes \Aut(\mathbb{F}_q)$, the group of semilinear transformations of $\mathbb{F}_q$. See \cite[Chapter 2]{Taylor}. We now prove Theorem \ref{main1bis}.
\begin{proof}[Proof of Theorem \ref{main1bis}]
Let $G$ be a finite primitive group and assume that $V:=\soc(G)$, the socle of $G$, is abelian. Then $V$ is elementary abelian, $V=\mathbb{F}_p^n$, $p$ a prime, and $G=V \rtimes G_0$ with $G_0$ acting irreducibly on $V$. Set $q:=p^n=|V|$, the primitivity degree of $G$. We assume that $G_0$ is MP and that any two elements of $V$ are conjugate or inverse-conjugate in $V$. The strategy of the proof is to reduce the analysis to a finite and small list of degrees and to check the corresponding primitive groups using GAP. By Proposition \ref{solvability}, the group $G_0$, and hence $G$, is solvable. Since $V$ is a minimal normal subgroup of $G$, if $0 \neq x \in V$ then $\langle x^G \rangle = V$, so if $y$ is any other nontrivial element of $V$, then $\langle y^G \rangle = V$ hence $y$ is conjugate to one of $x$, $-x$. This shows that $\{0\} \cup x^G \cup (-x)^G = V$. The group $G_0$ acts faithfully and irreducibly on $V$ with at most $3$ orbits. If $G$ acts transitively on the nonzero vectors of $V$ then $G$ is $2$-transitive. In this case, \cite[Table 7.3]{Cameron} implies that either $G_0$ is a subgroup of $\GammaL(1,q)$ or $q$ is in the following list: $2^2$, $3^2$, $3^4$, $5^2$, $7^2$, $11^2$, $23^2$. 
If $G$ is 2-transitive and of primitivity degree $q$ in  the above list, then $q\in \{2^2,3^2\}$, and  $G\in \{A_4,M(9)\}$ if $G$ is MP else $G\in \{S_4, \AGammaL(1,9)\}$, in particular $G_0\leqslant \GammaL(1,q)$.
We can therefore suppose that either $G$ has rank $3$, in other words $G$ admits $3$ orbits in its action on $V$, or $G_0 \leqslant \GammaL(1,q)$. We will discuss these two cases below.

Assume $G$ has rank $3$. Faulser and Kallaher \cite{F,FK} classified the maximal solvable finite primitive groups of rank $3$ and Liebeck \cite{L} classified all finite primitive groups of rank $3$. Since the two nontrivial orbits are inverse of each other, denoting their sizes by $a,b$ we have $a=b$. According to \cite[Theorem 1.1]{F}, if $G$ is not a collineation group of an affine line, then it either acts imprimitively on $V$ (item 3) or it falls into a finite list of exceptions (item 2). If $G$ does not act imprimitively on $V$ then, since $a=b$, we deduce that $G$ is primitive of degree $47^2$ and contained in a maximal primitive solvable group of degree $3$ of order $47^2 \cdot 24 \cdot 46$ and a GAP check shows that there is no group $G$ satisfying this. Assume now that $G$ acts imprimitively on $V$ and there is a decomposition $V = V_1 \oplus V_2$ into minimal imprimitivity subspaces of $V$ such that the two nonzero orbits are $V_1 \cup V_2-\{0\}$ and $V-(V_1 \cup V_2)$. The equality $a=b$ translates into $|V|=2(|V_1|+|V_2|-1)-1$. Reducing modulo $p$ we deduce that $p=3$ and one of $V_1$, $V_2$ has size $3$, so the other one also does (by imprimitivity). We deduce that $|V|=|V_1 \oplus V_2|=3^2=9$ hence $G$ has degree $9$. By a simple inspection we see that there is no group $G$ satisfying this. We successfully reduced to the case $G_0 \leqslant \GammaL(1,q)$.

Now assume $G_0$ is a subgroup of $\GammaL(1,q)$. Recall that $\GammaL(1,q)$ is the semidirect product $\mathbb{F}_q^{\ast} \rtimes \Aut(\mathbb{F}_q)$, in particular, if $H$ denotes the intersection between $G_0$ and $\mathbb{F}_q^{\ast}$, then both $H$ and $G_0/H$ are cyclic groups. It follows that $G_0/H$ is cyclic and MP, so $|G_0/H| \in \{1,2,3,4,6\}$, hence $G_0$ is an extension $C_m.C_{\ell}$ with $m,\ell$ positive integers and $\ell \leqslant 6$. If $t_1,t_2 \in C_m$ have order $m$ then their normal closure in $G_0$ is $C_m$ hence, since $G_0$ is MP, $t_1$ is conjugate to $t_2$ or to its inverse. This implies that $\varphi(m) \leqslant 2 \ell \leqslant 12$, where $\varphi$ denotes Euler's phi function, hence $m \leqslant 42$. It follows that 
$$q = |V| \leqslant 1+2|G_0| \leqslant 1+2 \cdot 42 \cdot 6 = 505.$$ 
We deduce that the primitivity degree of $G$ is at most $505$. We can now proceed by inspection using GAP.
\end{proof}

Given a finite solvable group $G$, recall that we denote its Fitting height by $h(G)$. Of course $h(G)=1$ if and only if $G$ is nilpotent. If $F(G)$ denotes the Fitting subgroup of $G$, $F_0(G)=\{1\}$ and $F_{i+1}(G)/F_i(G) = F(G/F_i(G))$ for $i \geqslant 0$, then $h(G)$ equals the smallest $n$ such that $F_n(G)=G$. Since $F(G/\Phi(G))=F(G)/\Phi(G)$, it follows that $h(G)=h(G/\Phi(G))$. We can now prove Corollaries \ref{Fitting} and \ref{consequences}.

\begin{proof}[Proof of Corollary \ref{Fitting}]
By Theorem \ref{main1} the Fitting height of a finite primitive MP group is at most $2$. The result now follows from the theorem and the fact that if $G$ is a finite solvable group then $h(G)=h(G/\Phi(G))$, moreover if $H \leqslant G$ then $h(H) \leqslant h(G)$, and $h(A \times B) = \max\{h(A),h(B)\}$ for any two finite solvable groups $A,B$.
\end{proof}

\begin{proof}[Proof of Corollary \ref{consequences}]
Item 1 follows from Theorem \ref{main1} and the fact that every prime divisor of the order of a finite group $G$ divides $|G/\Phi(G)|$ (this is an easy consequence of the Schur-Zassenhaus theorem and the fact that, by Frattini's argument, a Sylow $p$-subgroup of $\Phi(G)$ is normal in $G$). Items 2 and 3 follow from the theorem and the fact that a finite group $G$ is nilpotent if and only if $G/\Phi(G)$ is nilpotent.
\end{proof}

\section{Direct products and crown-based powers} \label{dirprod}

An important observation is that a finite direct product of MP groups is not necessarily MP. An easy example is given by the cyclic group $C_{12} \cong C_4 \times C_3$, which is a non-MP group isomorphic to a direct product of two MP groups (of coprime orders). One may also ask whether a direct power of an MP group is necessarily MP. Again, this is false, an example is given by $G \times G$, which is not MP, where $G = C_7 \rtimes C_3$. Therefore, Lemma \ref{quoMP} implies that, for this choice of $G$, the direct power $G^n$ is MP if and only if $n=1$. Not that, by Lemma \ref{quoMP}, a direct product of finite groups is not MP if one of the factors is not MP. 

Note that if $\pi:A \to B$ is a surjective group homomorphism and $a \in A$, then $\pi(a^A)=\pi(a)^B$, therefore $\pi \left( \langle a^A \rangle \right) = \langle \pi(a^A) \rangle = \langle \pi(a)^B \rangle$. We will apply this observation many times to the canonical projections of a direct product of groups on its factors.

We have the following results. 

\begin{prop}\label{powersdpMP}
Let $n$ be a positive integer and, for $i \in \{1,\ldots,n\}$, assume $G_i$ is a finite group equal to a semidirect product $N_i \rtimes H_i$ with $N_i \neq \{1\}$ and such that the following hold.
\begin{enumerate}
\item $N_i-\{1\}$ is a conjugacy class in $G_i$ for all $i \in \{1,\ldots,n\}$.
\item $C_{N_i}(h) = \{1\}$ for all $h \in H_i-\{1\}$ and all $i \in \{1,\ldots,n\}$. 
\item $L= \prod_{i=1}^n H_i$ is MP.
\end{enumerate}
Then $\Pi = \prod_{i=1}^n G_i$ is MP.
\end{prop}

\begin{proof}
First, note that $N_i$ is abelian for all $i$. Indeed, since $N_i-\{1\}$ is a conjugacy class in $G_i$, it is clear that $N_i$ is a $p_i$-group for some prime $p_i$, moreover the derived subgroup $N_i'$ is normal in $G$ (being characteristic in $N_i$). If $N_i'$ is nontrivial then it contains the whole conjugacy class $N_i-\{1\}$ so it equals $N_i$, however this is not possible since $N_i$ is a finite $p_i$-group. Since $L$ is MP, Lemma \ref{quoMP} implies that $H_i$ is also MP for all $i$. Write $K:=\prod_{i=1}^n N_i$ so that $\Pi = K \rtimes L$. Choose any $1 \neq u_i \in N_i$. Let $x,y \in \Pi$ have the same normal closure in $\Pi$. We need to show that $x$ and $y$ are conjugate or inverse-conjugate (in $\Pi$). Assume first that one of $x,y$ is contained in $K$, then of course both are contained in $K$ since $K$ is normal in $\Pi$. Without loss of generality, since $N_i-\{1\}$ is the conjugacy class of $u_i$ in $G_i$, we can assume that $x_i,y_i \in \{1,u_i\}$ for all $i \in \{1,\ldots,n\}$. It is easy to see that $x_i=1$ if and only if $y_i=1$, for all $i$. This implies that $x_i=y_i$ for all $i$. This shows that in this case $x$ and $y$ are conjugate.

Now assume that $x,y$ are not contained in $K$. Observe that every element of $G_i$ outside $N_i$ is conjugate to some element of $H_i$. Indeed, if $h \in H_i-\{1\}$ and $v \in N_i$ then the conjugate $h^v=v^{-1} h v$ is equal to $(-v+v^{h^{-1}})h$ and the set $\{-v+v^{h^{-1}}\ :\ v \in N_i\}$ is equal to $N_i$ because the map $N_i \to N_i$, $v \mapsto -v+v^{h^{-1}}$ is injective, hence surjective, since $C_{N_i}(h)=\{1\}$. This argument shows that, for all $i \in \{1,\ldots,n\}$, either both $x_i,y_i$ belong to $N_i$, in which case they are conjugate in $G_i$, or they don't, in which case they are conjugate in $G_i$ to elements of $H_i$. Therefore, calling $I$ the set of indices such that both $x_i$ and $y_i$ do not belong to $N_i$, we may assume that $x_i,y_i \in H_i$ for all $i \in I$. Call $\pi_I$ the canonical projection $\Pi \to H_I$, where $H_I = \prod_{i \in I} H_i$. Then $\pi_I(x)$ and $\pi_I(y)$ generate the same normal closure in $H_I$, which is MP (it is a quotient of $L$, which is MP), hence there exists $l_i \in H_i$, for all $i \in I$, such that $y_i^{\varepsilon} = l_i^{-1} x_i l_i$, where $\varepsilon \in \{1,-1\}$ is independent of $i$. If $i \not \in I$, then $x_i,y_i$ are conjugate and, at the same time, inverse-conjugate in $G_i$, so there exists $l_i \in H_i$ such that $y_i^{\varepsilon} = l_i^{-1} x_i l_i$. Let $l=(l_1,\ldots,l_n)$. Then $y^{\varepsilon} = l^{-1}xl$. Hence in this remaining case $x$ and $y$ are conjugate or inverse-conjugate.
\end{proof}

\begin{prop}\label{SMPclass} Let $A,B$ be two finite groups. The following assertions hold.
\begin{enumerate}
\item If $B$ is SMP, then $A \times B$ is MP if and only if $A$ is MP.
\item $A\times B$ is SMP if and only if $A$ and $B$ are SMP.
\end{enumerate}
\end{prop}

\begin{proof}
Item (2) follows from item (1) because finite SMP groups are closed under taking quotients and a finite MP group is SMP if and only if every element is conjugate to its own inverse. So we only prove item (1). If $A\times B$ is MP then $A$ is MP by Lemma \ref{quoMP}.
Conversely assume that $A$ is MP.  
We need to show that $A \times B$ is MP. 
Let $(a_1,b_1)$ and $(a_2,b_2)$ be elements of $A \times B$ such that 
\begin{equation}\label{e:dpsmpa1}
\langle (a_1,b_1)^{A \times B}\rangle=\langle (a_2,b_2)^{A \times B} \rangle.\\
\end{equation}
Then $\langle a_1^A \rangle = \langle a_2^A \rangle$ and $\langle b_1^B \rangle=\langle b_2^B \rangle$. Therefore there exist $r \in A$, $s \in B$ such that $a_2=(a_1^{\pm 1})^r$ and $b_2=b_1^s$. Since $B$ is SMP, we can write $b_1=(b_1^{-1})^h$ for some $h \in B$. We either have $(a_2,b_2) = (a_1^r,b_1^s) = (a_1,b_1)^{(r,s)}$ or $(a_2,b_2)=((a_1^{-1})^r,b_1^s)=(a_1^{-1},b_1^{-1})^{(r,hs)}$. Therefore $(a_1,b_1)$ and $(a_2,b_2)$ are conjugate or inverse-conjugate, and so $A \times B$ is MP.
\end{proof}

For example the groups of the form $C_2^k \times C_3^n \times S_3^m$ are MP and they are SMP if $n=0$. The above proposition establishes that a direct product of finite groups is SMP if and only if all the factors are SMP. For MP groups, the situation is much more complicated. 
However we can prove Proposition \ref{MPdir}.

\begin{proof}[Proof of Proposition \ref{MPdir}]
We consider item (2) as item (1) is immediate from Lemma \ref{quoMP}. By Theorem \ref{main1}, for $1\leqslant i \leqslant n$, $G_i$ is one of the following groups. 
$$C_2,\ C_3,\ S_3,\ A_4,\ \mathrm{AGL}(1,5),\ C_7\rtimes C_3,\ \mathrm{AGL}(1,7),\ M(9).$$
Observe that the finite primitive MP groups appear as semidirect products
$$C_2=C_2 \rtimes 1,\ C_3=C_3 \rtimes 1,\ S_3=C_3 \rtimes C_2,\ A_4=(C_2\times C_2) \rtimes  C_3,\ \mathrm{AGL}(1,5)=C_5 \rtimes C_4,$$ 
$$C_7\rtimes C_3, \ \mathrm{AGL}(1,7) = C_7 \rtimes (C_2\times C_3),\ M(9)=(C_3 \times C_3) \rtimes Q_8.$$
By Theorem \ref{main1}, $C_2$ and $M(9)$ are SMP and so Proposition \ref{SMPclass} yields that $C_2^n\times M(9)^\ell$ is SMP for any nonnegative integers $n$ and $\ell$. Since $C_3^m$ and $C_4^m$ are MP for any nonnegative integer $m$, Proposition \ref{SMPclass} gives that $C_2^n \times C_3^m\times M(9)^\ell$ and $C_2^n \times C_4^m \times M(9)^\ell$ are MP for any nonnegative integers $n$, $m$ and $\ell$.

If none of the $G_i$'s are $C_7\rtimes C_3$ or $\mathrm{AGL}(1,5)$ then $G$ is MP by Proposition \ref{powersdpMP}.

Suppose that one of the $G_i$'s is $C_7\rtimes C_3$, without loss of generality $G_1=C_7\rtimes C_3$. Note that by Lemma \ref{quoMP}$, (C_7\rtimes C_3)\times \mathrm{AGL}(1,5)$ is not MP, as it has a quotient isomorphic to $C_3\times C_4$ which is not MP. Using GAP, one checks that $G_1\times G_2$ is MP if and only if $G_2$ is SMP.  It now follows from Proposition \ref{SMPclass} that $G$ is MP if and only if $G_j$ is SMP for all $j = 2, \ldots, n$. 

Suppose finally that none of the $G_i$'s are $C_7\rtimes C_3$ and that one of them is $\mathrm{AGL}(1,5)$, say $G_1=\mathrm{AGL}(1,5)$.   If $G_j\in \{C_3, A_4, \mathrm{AGL}(1,7)\}$  for some $2\leqslant j\leqslant n$ then $G_1\times G_j$ has a quotient isomorphic to $C_3\times C_4$ which is not MP, and so by Lemma \ref{quoMP} $G$ is not MP. Finally if $G_j\not \in \{C_3, A_4, \mathrm{AGL}(1,7)\}$  for  $2\leqslant j\leqslant n$ then $G$ is MP by Proposition \ref{powersdpMP}.  
\end{proof}

By \cite{Baer} a finite primitive solvable group is monolithic. We recall some facts about crown-based powers of a finite monolithic primitive group $L$. Let $V:=\mathrm{soc}(L)$ be the socle of $L$. For every positive integer $k$ one has
$$L_k=\{(l_1,\ldots, l_k)\in L^k: l_1 \equiv \ldots \equiv l_k \mod V \}= V^k \mathrm{diag}(L^k)$$
where $\mathrm{diag}(L^k)=\{(l_1,\ldots, l_k)\in L^k: l_1 = \ldots = l_k\}\cong L$.
We have $\mathrm{soc}(L_k)=V^k$, $L_k/V^k\cong L/V$, and if $k>1$ then the quotient  $L_k/U$ of $L_k$ by any minimal normal subgroup $U$ of $L_k$ satisfies $L_k/U\cong L_{k-1}$. Also every normal subgroup $N$ of $L_k$ either satisfies $N\leqslant V^k$ or $V^k<N$. See also \cite[page 657]{DL}.

Suppose  $L$ is solvable. The socle $V$ is then complemented in $L$ and we write $L=VH$ where $H$ is a complement of $V$ in $L$. In this case, $L_k= V^k\rtimes \mathrm{H}$ and we can therefore think of elements of $L_k$ as products $vh$ where $v\in V^k$ and $h \in H$. In other words, $L_k$ is the semidirect product $V^k\rtimes H$ where $H$ acts coordenatewise and in the same way on all the factors.

We can now prove Proposition \ref{MPcrownbased}.

\begin{proof}[Proof of Proposition \ref{MPcrownbased}]
Suppose that $L_k$ is (S)MP. Since $L_{k-1}$ is a quotient of $L_k$, we deduce that $L_i$ is (S)MP for all $i$ with $1 \leqslant i \leqslant k$. If $k=1$ then, by Theorem \ref{main1}, $L$ is indeed one the the eight finite primitive MP groups listed therein. If $L=M(9)$ then $L_2$ is not MP, so $L_k$ is MP if and only if $k=1$. If $L\in \{C_2,C_3\}$ then, for every $k>0$, $L_k=L^k$ and so $L_k$ is MP. In particular if $L=C_2$ then $L_k$ is SMP for every $k>0$. Suppose now that $L$ is isomorphic to one of the following groups:
$$S_3,\ A_4,\ \mathrm{AGL}(1,5),\ C_7\rtimes C_3,\ \mathrm{AGL}(1,7).$$
It remains to show that $L_k$ is MP for every $k>1$. Note that $L$ is a Frobenius group, therefore $L_k$ is a Frobenius group for all $k \geqslant 1$, in particular the conjugacy classes of $L_k$ not contained in $V^k$ are precisely the cosets of $V^k$ in $L_k$. Since $H$ is (S)MP, this implies that two elements of $L_k$ outside $V^k$ that generate the same normal subgroup are either conjugate or inverse-conjugate (and they are conjugate if $L$ is SMP). Assume now that $x,y$ are nontrivial elements of $V^k$ that generate the same normal subgroup of $L_k$.

Assume $L$ is not isomorphic to $A_4$. Then $L=\langle a,b: a^p = b^n = 1, bab^{-1} = a^q \rangle$ for some some prime $p$ and some positive integers $n$ and $q$ where $n$ divides $p-1$, $q<p$ and $q$ has order $p-1$ modulo $p$. In particular $V=\langle a\rangle$ and $H=\langle b \rangle$. Our assumption on $L$ implies that every subgroup of $V^k$ is normal in $L_k$, therefore $\langle x\rangle=\langle x \rangle^{L_k}=\langle y\rangle^{L_k}=\langle y\rangle$. Since $x,y$ have order $p$, there is a positive integer $r$ coprime to $p$ such that $y=x^r$. There exists $h \in H$ that induces one of the two automorphisms $t \mapsto t^r$, $t \mapsto t^{-r}$ on $V$ (in the case $L=C_7 \rtimes C_3$, we can choose only one of these two). It follows that $y=(x^{\pm 1})^h$. This argument also proves that $L_k$ is SMP if $L$ is SMP.

If $L \cong A_4$ then $x,y$ have order $2$ and $H \cong C_3$. Any product $x^{h_1} \ldots x^{h_t}$ (with $h_i \in H$ for all $i$), if not equal to $1$, has the form $x^h$ for some $h \in H$. This is because we can think of $H$ as being generated by the matrix $u = \left( \begin{array}{cc} 0 & 1 \\ 1 & 1 \end{array} \right)$ over $\mathbb{F}_2$, satisfying $u^2+u=1$, so that any sum of powers of $u$ is either $0$ or a power of $u$. Therefore 
$$x^{u^{i_1}} \ldots x^{u^{i_t}} = x^{u^{i_1}+\ldots+u^{i_t}} \in \{x,x^u,x^{u^2}\} = x^{L_k}.$$ 
This proves that $\langle x \rangle^{L_k} \setminus \{1\} = x^{L_k}$ and the same holds for $y$, so that $x,y$ are conjugate in $L_k$. We have now established that $L_k$ is MP.
\end{proof}

\section{$p$-rank of finite MP groups} \label{sectionprank}

We have described the structure of $G/\Phi(G)$ when $G$ is a finite MP group, however there are still many questions open, due to the fact that many properties of the Frattini quotient of a group $G$ are not carried over to $G$. In this section we prove that, in the case of finite MP groups, the $p$-rank does not change when passing to the Frattini quotient.

\begin{lemma}\label{l:power23}
Let $m=2^a$ and $n=3^b$ where $a, b$ are nonnegative integers. Then $|m-n|=1$ if and only if $(a,b)\in \{(1,0),(1,1),(2,1),(3,2)\}$.
\end{lemma}

This is fairly well-known but we include a proof.

\begin{proof}
We prove the ``only if'' statement as the converse is obvious. Assume first that $3^b = 2^a-1$. If $a=2v$ is even, then $3^b = (2^v-1)(2^v+1)$ hence both $2^v-1$ and $2^v+1$ are powers of $3$, implying that $(a,b)=(2,1)$. If $a$ is odd then $3^b = 2^a-1 \equiv 1 \mod 3$ implying $(a,b)=(1,0)$. Assume now that $2^a = 3^b-1$. If $b=2u$ is even, then $2^a = 3^b-1 = (3^u-1)(3^u+1)$ hence both $3^u+1$ and $3^u-1$ are powers of $2$, so $u=1$ and $(a,b)=(3,2)$. If $b$ is odd, then $2^a = 3^b-1 \equiv 2 \mod 4$ implying that $(a,b)=(1,1)$.
\end{proof} 

The proof of the following lemma is essentially the same as the proof of \cite[Theorem 3.2]{K}.

\begin{lemma} \label{frattinicf}
Let $G$ be a finite solvable group and let $N$ be a minimal normal subgroup of $G$ contained in the Frattini subgroup of $G$. Let $p$ be the prime number dividing $|N|$ and let $K$ be a normal subgroup of $G$ containing $N$. If $K$ centralizes all chief $p$-factors of $G$ between $N$ and $K$, then $K$ centralizes $N$.
\end{lemma}
\begin{proof}
Assume by contradiction that $K$ centralizes all chief $p$-factors of $G$ between $N$ and $K$ and that $K$ does not centralize $N$. Note that a nontrivial element $x$ of $N$ cannot be central in $G$, as otherwise $\langle x\rangle$ would be a normal subgroup of $G$ properly contained in $N$. Let $C:=C_G(N)$, it does not contain $K$ but, as $N$ is abelian, it contains $N$. Let $D = K \cap C$ so that $N \leqslant D < K$. Let $E/D$ be a minimal normal subgroup of $G/D$ contained in $K/D$. Then $E \leqslant K$, hence $E \cap C = K \cap C = D$, therefore $E/D$ is $G$-isomorphic to  $CE/C$, which therefore is a minimal normal subgroup of $G/C$. The group $N$ can be seen as an $\mathbb{F}_p$-vector space. By Clifford's theorem, a finite group acting faithfully and irreducibly on a finite $\mathbb{F}_p$-vector space does not have any nontrivial normal $p$-subgroup, since such a subgroup would have to act completely reducibly, hence trivially \cite[A. Lemma 13.6]{DH}. So $G/C$ does not have any nontrivial normal $p$-subgroup. It follows that $CE/C$ and hence $E/D$ has order coprime to $p$. 
Since $K$ is solvable, it admits a Hall $p'$-subgroup $Q$ (also known as a Sylow $p$-complement), in other words $Q$ is a subgroup of $K$ such that $|Q|$ is coprime to $p$ and $|K:Q|$ is a power of $p$. Since by Hall's theorem all Hall $p'$-subgroups are conjugate, and $D,E$ are normal subgroups of $K$, it follows that $E \leqslant DQ$. 

By assumption, $K$ centralizes all chief $p$-factors of $G$ in $K/N$, so the chief $p$-factors of $K/N$ are central in $K/N$. It follows from \cite[A. Theorem 13.8(a)]{DH} that $K/N$ is $p$-nilpotent, in other words the Sylow $p$-complement $NQ/N$ of $K/N$ is normal in $K/N$. As $NQ/N$ is a normal Sylow $p$-complement of $K/N$, it follows by Hall's theorem that $NQ/N$ is the unique Sylow $p$-complement of $K/N$ and so it is characteristic in $K/N$. Since $K$ is a normal subgroup of $G$, we deduce that $NQ \unlhd G$. Let $H$ be the normalizer of $Q$ in $G$. Since $NQ \unlhd G$ and the Hall $p'$-subgroups of $NQ$ are all conjugate by Hall's theorem, if $x \in G$ then there exists $y \in NQ$ such that $Q^x = Q^y$, so that $xy^{-1} \in H$ implying that $x \in Hy \subseteq HNQ = HN$. We deduce that $HN = G$. Since $N$ is contained in the Frattini subgroup of $G$, we must have $H=G$, in other words $Q \unlhd G$. Since $|Q|$ is coprime to $p$, the commutator $[N,Q]$ is equal to $\{1\}$ hence $Q \leqslant C$. Therefore $E \leqslant DQ \leqslant C$, a contradiction, as required.
\end{proof}

Let $p$ be a prime number, $G$ a finite solvable group. Recall that $r_p(G)$ (the $p$-rank of $G$) is the maximal rank of a chief $p$-factor of $G$. In other words $r_p(G)$ is the largest nonnegative integer $k$ such that $G$ has chief factors of order $p^k$. Recall also that $S_p(G)$ denote the largest positive integer $s$ such that $G$ has a maximal subgroup of index $p^s$ if $p$ divides $|G|$, and $S_p(G)=0$ otherwise. Recall that, in a finite solvable group, every maximal subgroup has prime power index. Of course $S_p(G) = S_p(G/\Phi(G))$. The group $G$ is supersolvable if and only if $r_p(G)=1$ for every prime divisor $p$ of $|G|$. By \cite[Theorem 3.2(1)]{K}, if $S_p(G) = 1$ then $r_p(G) = 1$, implying that $G$ is supersolvable if and only if $G/\Phi(G)$ is supersolvable. However, the numbers $r_p(G)$, $r_p(G/\Phi(G))$ are not necessarily equal in general.

It is interesting to point out that, by \cite[Satz 1]{Huppert}, $r_p(G)$ is equal to $j_p(G)$, the largest integer $j$ such that $p^j$ appears as an index in a maximal chain of subgroups of $G$. For example, the group $A_4$ has two distinct maximal chain types, the indices are $2,2,3$ in one, $3,4$ in the other one, so $j_2(A_4)=2$ and the maximal index that appears in a maximal chain depends on the chain.

\begin{lemma} \label{srineq}
If $G$ is a finite solvable group and $p$ is a prime dividing $|G|$, then $1 \leqslant S_p(G) \leqslant r_p(G)$. Moreover, $r_p(G)=1$ if and only if $r_p(G/\Phi(G))=1$.
\end{lemma}
\begin{proof}
Hall's theorem implies that $S_p(G) \geqslant 1$. Let $s=S_p(G)$ and let $M$ be a maximal subgroup of $G$ of index $p^s$. Let $M_G$ be the normal core of $M$ in $G$ and let $N/M_G$ be a minimal normal subgroup of $G/M_G$. Since every normal subgroup of $G$ contained in $M$ is contained in $M_G$, we have that $M$ does not contain $N$, hence $MN=G$. Since $N/M_G$ is abelian, $(M \cap N)/M_G$ is normal in $N/M_G$, and it is of course normal in $M/M_G$ as well, therefore $M \cap N \unlhd MN = G$. By definition of $N$, we deduce that $M \cap N = M_G$, therefore $p^s = |G:M| = |N/M_G| \leqslant p^r$ where $r=r_p(G)$, since $N/M_G$ is a chief factor of $G$. Therefore $s \leqslant r$.

Now, if $r_p(G)=1$ then of course $r_p(G/\Phi(G))=1$ since the chief factors of $G/\Phi(G)$ are, in particular, chief factors of $G$. Conversely, if $r_p(G/\Phi(G))=1$, then $S_p(G) = S_p(G/\Phi(G)) \leqslant r_p(G/\Phi(G))=1$ implies that $S_p(G)=1$ hence $r_p(G)=1$ by \cite[Theorem 3.2(1)]{K}.
\end{proof}

We can now prove Theorem \ref{thmprank}.

\begin{proof}[Proof of Theorem \ref{thmprank}]
Note that if a group $A$ is a subdirect product of a family of groups $G_i$, $i\in I$, then every chief factor of $A$ is isomorphic to a chief factor of one of the $G_i$, therefore $r_p(A)=\mathrm{max}\{r_p(G_i): i \in I\}$. We will use this fact as well as the one that $S_p(G)=S_p(G/\Phi(G))$. Using Theorem \ref{main1}, Lemma \ref{srineq}, the fact that if $S_p(G)=1$ then $r_p(G)=1$, and the fact that $G$ is supersolvable if and only if $G/\Phi(G)$ is supersolvable, it is now easy to prove item (1) and the fact that item (4) follows from items (1), (2), (3).
For item (2), Theorem 3.1 and Lemma 5.3 give
$$S_p(G)=S_p(G/\Phi(G))\leqslant r_p(G/\Phi(G))\leqslant 2$$
since $G/\Phi(G)$ is a subdirect product of finite primitive MP groups. If $S_p(G)=1$ then $r_p(G)=1$, so $r_p(G/N)\leqslant 1$ for any $N\unlhd G$, and so $G$ has no quotient isomorphic to $A_4$ nor $M(9)$. Noting that $r_p(G/\Phi(G))=1$ if and only if $G/\Phi(G)$ has no quotient isomorphic to $A_4$ if $p=2$, $M(9)$ if $p=3$ yields item (2). 

We now consider item (3). Assume that $p \in \{2,3\}$. We prove that $r_p(G) \leqslant 2$ by induction on $|G|$. Let $K$ be the intersection of the normal subgroups $J$ of $G$ such that $G/J$ is isomorphic to an irreducible subgroup of $\GL(1,p)$ or $\GL(2,p)$. Let $I_{(2,p)}$ be the set of irreducible subgroups of $\mathrm{GL}(2,p)$ that are MP. Then
$$I_{(2,p)}=\left\{\begin{array}{ll} 
\{C_3, S_3\} & \textrm{if} \ p=2\\

\{C_4, D_8, Q_8, QD_{16}\} & \textrm{if} \ p=3. 
\end{array}
\right.
$$
Since the class of finite supersolvable groups form a pseudovariety of finite groups, it follows that $G/K$ is a supersolvable group of exponent dividing $6$ if $p=2$, and a 2-group of exponent dividing $8$ if $p=3$. If $K=\{1\}$ then $G$ is supersolvable hence $r_p(G)=1$. Now assume $K \neq \{1\}$. Let $N$ be a minimal normal subgroup of $G$ contained in $K$. Write $|N| = q^s$ for some prime $q$ and some positive integer $s$, so that $N \cong \mathbb{F}_q^s$. As a quotient of a finite MP group, $G/N$ is a finite MP group, so by induction $r_p(G/N) \leqslant 2$. Hence $r_p(G) \leqslant 2$ unless $q=p$ and $s>2$.

We therefore assume that $|N|=p^s$ where $s>2$. Let $C=C_G(N)$. As $N$ is not cyclic, note that $C\neq G$, otherwise every normal subgroup of $N$ is normal in $G$ contradicting the minimality of $N$. We now show that $C$ does not contain $K$. Suppose for a contradiction that $C$ contains $K$. Then $G/C$ is a factor group of $G/K$, and so $G/C$ is a supersolvable group of exponent dividing $6$ if $p=2$, otherwise $G/C$ is a 2-group. If $N \setminus \{1\}$ is a unique conjugacy class in $G$, then for $1 \neq x \in N$ we have $|G:C_G(x)|=|N|-1=p^s-1$. But since $C_G(x) \geqslant C \geqslant K$, it follows that either $p=2$ and $p^s-1$ is a power of $3$ or $p=3$ and $p^s-1$ is a power of $2$. Lemma \ref{l:power23} now implies that $s\leqslant 2$, a contradiction. This proves that $N$ contains at least two nontrivial conjugacy classes of $G$. We now claim that there is a conjugacy class in $G$ of nontrivial elements of $N$, say with representative $x$, such that $x$ is conjugate to $x^{-1}$. This is clear if $p=2$ since in this case $N$ is an elementary abelian $2$-group so every element of $N$ equals its own inverse. Suppose now that $p=3$. The nontrivial $2$-group $G/C \leqslant \GL(s,3)$ has involutions, and any involution of $\GL(s,3)$ has minimal polynomial dividing $X^2-1$, so it has $-1$ as an eigenvalue. A corresponding eigenvector $x \in N$ is therefore conjugate in $G$ to $x^{-1}$.
We have now established the above claim.
We also know that there is at least a further conjugacy class in $G$ of nontrivial elements of $N$, say with representative $y$. Since $N$ is a minimal normal subgroup of $G$,  $\langle x^G\rangle = \langle y^G\rangle = N$. However $y$ is not conjugate in $G$ to $x$ nor $x^{-1}$, contradicting the fact that $G$ is MP. This contradiction establishes the fact that $C$ does not contain $K$.

We are in the following situation: the normal subgroup $K$ of $G$ contains $N$ and centralizes all chief $p$-factors of $G$ above $N$ by definition of $K$, since $r_p(G/N) \leqslant 2$, and $K$ does not centralize $N$. By Lemma \ref{frattinicf}, $N$ is not contained in the Frattini subgroup of $G$, therefore there exists a maximal subgroup $H$ of $G$ such that $NH=G$. The intersection $N \cap H$ is normal in $N$ since $N$ is abelian and it is normal in $H$ since $N$ is normal in $G$, therefore $N \cap H = \{1\}$ by minimality of $N$. It follows that $|G:H| = |N| = p^s$ and this contradicts the fact that $S_p(G) \leqslant 2$.\\

We finally consider item (5). By items (1)-(3), $r_p(G)$ and $S_p(G)$ belong to $\{1,2\}$. If $r_p(G)=1$ then $S_p(G)=1$ by Lemma \ref{srineq}. Suppose $r_p(G)=2$ and assume for a contradiction that $S_p(G)=1$. Then by items (1) and (2), $p\in \{2,3\}$ and $G$ has no quotients isomorphic to $A_4$ if $p=2$, $M(9)$ if $p=3$. Since 
$$\{H \ \textrm{finite group}: r_p(H)\leqslant 1\}$$
is a pseudovariety of finite groups, 
it follows from Theorem \ref{main1} that $r_p(G/\Phi(G))=1$. Item (4) now yields $r_p(G)=1$, a contradiction. This establishes item (5).
\end{proof}

\begin{proof}[Proof of Corollary \ref{c:orderchief}]
If $M$ is a chief factor of $G$ then $U = M \rtimes G/C_G(M)$ is a primitive solvable group with socle $M$ such that $U/M$ is MP and any two nontrivial elements of $M$ are either conjugate or inverse-conjugate in $U$. However, note that $U$ is not necessarily isomorphic to a quotient of $G$. If $U$ is MP, then it is one of the groups in Theorem \ref{main1} and the result follows in this case. If $U$ is not MP then, by Theorem \ref{main1bis}, $U$ is one of the following groups: $S_4$, $\AGammaL(1,8)$, $\AGammaL(1,9)$, $\AGammaL(1,16)$, so that $|M| \in \{4,8,9,16\}$. Now, Theorem \ref{thmprank} implies that $8$ and $16$ are not orders of chief factors of a finite MP group, hence the result follows.
\end{proof}

\end{document}